\documentclass[11pt]{article}
\usepackage{preamble}
\usepackage{times}
\usepackage{mathtools}
\newcommand{\Bin}{\operatorname{Bin}}
\newcommand{\LP}{\operatorname{LP}}

\newcommand{\Fc}{\F_{\smash{\sY|\sX}}}

\usepackage[backend=biber, maxbibnames=99]{biblatex}
\AtEveryBibitem{
    \clearfield{issn}
    \clearfield{urlyear}
    \clearfield{urlmonth}
    \clearfield{month}
    \clearfield{language}
}

\newcommand{\footremember}[2]{
   \footnote{#2}
    \newcounter{#1}
    \setcounter{#1}{\value{footnote}}
}
\newcommand{\footrecall}[1]{
    \footnotemark[\value{#1}]
}

\newcommand{\ul}{\underline}

\begin{document}

\title{Generative Modelling via Quantile Regression}

\author{%
    Johannes Schmidt-Hieber\footremember{alley}{University of Twente}\footremember{trailer}{The research has been supported by the NWO/STAR grant 613.009.034b and the NWO Vidi grant
VI.Vidi.192.021.}%
    \ and Petr Zamolodtchikov\footrecall{alley} \footrecall{trailer}%
}

\date{\today}

\maketitle

\begin{abstract}
    We link conditional generative modelling to quantile regression. We propose a suitable loss function and derive minimax convergence rates for the associated risk under smoothness assumptions imposed on the conditional distribution. To establish the lower bound, we show that nonparametric regression can be seen as a sub-problem of the considered generative modelling framework. Finally, we discuss extensions of our work to generate data from multivariate distributions.
\end{abstract}

\section{Introduction}
\label{sec.intro}

Generative modelling emerged as its own subfield of artificial intelligence with a series of successes \cite{kingma2013auto, rezende2015variational, goodfellow2020generative} including the recent introduction of diffusion models \cite{song2020score} and ChatGPT \cite{NIPS20173f5ee243}. These techniques enable computers to produce realistic fake pictures or human-like textual interactions. The potency and novelty of the resulting tools come with an equal amount of complexity, as is suggested by a certain lack of consensus---see, e.g.\ \cite{stanczuk2021wasserstein}, where the authors question the well-foundedness of Wasserstein GANs \cite{arjovsky2017wasserstein}. This pinpoints a need for a more theoretical understanding of generative modelling in all its aspects. Informally, generative modelling corresponds to the following task. Given samples from a distribution $\P,$ the aim is to produce a surrogate distribution $\wh \P$ such that samples from $\wh \P$ are indistinguishable from the original ones.\\

Inverse transform sampling provides a simple yet effective tool for sampling from a univariate distribution $\P_{\sY}$ whenever the associated cumulative distribution function (cdf) $\F_{\sY}$ admits a left-inverse. If $U$ is uniform on the unit interval, then 
\begin{align*}
    \F_{\sY}^{-1}(U) \sim \P_{\sY},
\end{align*}
with $\F_{\sY}^{-1}$ the quantile function. Given a covariate $X,$ approximating samples from the distribution of $Y|X$ will be called \emph{conditional generative modelling} and can be achieved through the conditional quantile function. Estimating the conditional quantile function is commonly referred to as quantile regression. Formally, the aim is to reconstruct
\begin{align}
    \label{eq.quantile.function}
    \Fc^{-1}(\tau|x) \colon (\tau, x) \mapsto \inf\big\{t \in \RR : \Fc(t|x) \geq \tau\big\}
\end{align}
from an i.i.d.\ sample $S_n:=\{(X_1, Y_1), \dots, (X_n, Y_n)\}$ with the same distribution as $(X,Y).$ Here, $\tau$ is called the quantile level and $\Fc$ is the cdf of $Y|X,$
\begin{align*}
    \Fc(t|x) := \PP\big\{Y\leq t | X = x\big\}.
\end{align*}
While regression focuses on the conditional mean given the covariates, quantile regression, introduced by \cite{koenk} in 1978, also captures the effects of the covariates on the entire distribution of the response variable. This results in a finer understanding of the regressors' impact on the outcome and makes quantile regression an attractive tool in various fields of research such as economics \cite{MARROCU201513}, clinical research \cite{cleophas2022quantile}, or environmental science \cite{cade2003gentle}. \\

From a theoretical standpoint, inferential aspects of quantile regression are well understood. Asymptotic normality is obtained in various settings \cite{koenker2005quantile, abrevaya2005isotonic, dette2008non}, and, more recently, analyses have been carried out under modern settings such as trend filtering \cite{madrid2022risk}, high-dimensional sparse models \cite{MR2797841} and ReLU networks regression \cite{padilla2022quantile}. Much of the research has focused on estimating the quantile function at a finite set of fixed levels, which, although sufficient for most applications, is unsuitable for generative purposes—obtaining good quality samples from the conditional distribution $\P_{\sY|\sX}$ requires reconstruction of the whole function $\Fc^{-1}.$ Few research has given insights into this aspect. \cite{wang2023generative} suggests using deep learning to fit the entire conditional quantile function. They show empirically that such an approach is promising. However, more theory is required for a proper understanding of their results. \cite{MR4047593} and \cite{MR2797841} obtained uniform rates in $(\tau, x)$ for quantile regression. In particular, \cite{MR4047593} consider the setting of isotonic regression with smoothness assumptions. They achieve rates comparable to nonparametric regression uniformly over $\tau$ in an increasing sequence of compact subsets of $(0, 1).$ The choice of these subsets mitigates the inconvenient property that the quantile function is typically divergent at zero and one, and a finite sample is insufficient to make a good guess around the boundary values $0$ and $1.$ Other notable results are obtained in \cite{volgushev2019distributed} for approximately linear models and a fixed range of quantile levels.\\

We study the generative properties of quantile regression estimates. In Section \ref{sec.2}, we introduce a loss function $\mL$ that averages the error over all quantile levels, circumventing the above mentioned thresholding issue. We then introduce the estimator and our working assumptions in the second part of Section \ref{sec.2}. Upper bounds are provided for a smoothed conditional cdf estimator for quantile regression over a natural extension of H\"older classes in Section \ref{sec.3}. We then show that nonparametric regression under smoothness assumptions can be seen as a sub-problem of quantile regression under the $\mL$ loss, implying the minimax optimality of the derived rates. The theoretical contributions are closely related to the very recent work \cite{PIC20231564}. For a more in-depth discussion see Section \ref{sec.3}. A short discussion of possible extensions is finally provided in Section \ref{sec.discussion}. All the proofs are deferred to the appendix.\\

\textbf{Notation.} Let $\|x\| := \max_{1\leq i \leq d}|x_i|$ be the infinity norm on $\RR^d$ and let $\B(x,r)$ denote the closed ball centered in $x$ with radius $r$ in the infinity norm. $\|\cdot\|_2$ denotes the Euclidean norm. For a function $f \colon A \to \RR,$ we write $\|f\|_{\infty} := \sup_{x \in A} |f(x)|.$  For a symmetric $k\times k$ matrix $M,$ denote by $\lambda_1(M) \leq \dots \leq \lambda_k(M)$ its ordered eigenvalues. The uniform distribution on the unit interval is denoted by $\mU,$ and $\mN(\mu,\sigma^2)$ denotes the normal distribution with mean $\mu$ and variance $\sigma^2.$ We make use of multi-index notations. For two integers $\ell$ and $d,$ let $m := \card\{\balpha \in \{0, \dots, \ell\}^d: |\balpha| \leq \ell\}.$ For $\balpha \in \{0, \dots, \ell\}^d,$ let
\begin{align*}
    \partial^{\balpha} := \prod_{j=1}^d \partial_j^{\alpha_j},
\end{align*}
where $\partial_j$ denotes the partial derivative with respect to the $j$-th coordinate. Additionally, $|\balpha| := \sum_{i=1}^d \alpha_i, \ \balpha ! := \alpha_1 !\dots \alpha_d !,$ and $x^{\balpha} = \prod_{i=1}^d x_i^{\alpha_i}.$ For any vector indexed in $\balpha \in \{0, \dots, \ell\}^d,$ the coordinates are ordered lexicographically on $\{0, \dots, \ell\}^d.$ If $f$ is a function of two variables $(x, t) \in \RR^d \times \RR,$ then, the differential operators are applied to the functions $x \mapsto f(x, t)$ or $x \mapsto f(t|x)$ with fixed $t.$

\section{Problem setup}
\label{sec.2}
    We assume that $X \sim \P_{\sX}$ is supported on $[0, 1]^d$ and $Y \sim \P_{\sY},$ where $\P_{\sY}$ is atomless and univariate. Conditional generative modelling aims to transport a simple distribution, say the uniform distribution $\mU,$ towards $\P_{\sY|\sX}.$ More precisely, the goal is to approximate a measurable function $T$ such that the pushforward of $\mU$ through $T$ is $\P_{\sY|\sX}.$ That is, $T(U|X) \sim \P_{\sY|\sX},$ for $U \sim \mU.$ Such a function is called a transport map. As mentioned in Section \ref{sec.intro}, the conditional quantile function \eqref{eq.quantile.function} is a suitable transport map. This links conditional generative modelling to quantile regression.
    \subsection{Model and loss}
    If we know $\Fc^{-1},$ we can generate $Y$ by sampling $U$ uniformly on $[0,1]$ and interpreting $\Fc^{-1}$ as a conditional generative model, that is,
    \begin{align}
    \label{eq.generative.model}
        Y = \Fc^{-1}(U|X).
    \end{align}
    For unknown $\Fc^{-1},$ we first construct an estimator $\smash{\wh \F}^{-1}_{\smash{\sY|\sX}}$ from the sample $S_n=\{(X_1,Y_1),\ldots,(X_n,Y_n)\}$ and obtain instead
    \begin{align*}
        \wh Y = \wh Y(S_n) := \smash{\wh \F}^{-1}_{\smash{\sY|\sX}}(U|X).    
    \end{align*}
    To measure the fidelity of the generated variable, we introduce the loss 
    \begin{align}
    \label{eq.loss}
        \mL(\wh Y, Y) := \E\big[|\wh Y - Y|\big|S_n\big],
    \end{align}
    and the associated risk is
    \begin{align*}
        \E\big[|\wh Y - Y\big|\big].
    \end{align*} 
    Plugging the expressions of $\wh Y$ and $Y$ into \eqref{eq.loss} shows that $\mL$ is the $L^1( \mU\otimes \P_{\sX})$ distance between the conditional quantile functions, which corresponds to an averaged Wasserstein-$1$ distance between the conditional distributions. Finally, it is known that, in the univariate setting, the Wasserstein-$1$ distance can be expressed as the $L^1$ distance between the cdfs (see, e.g., \cite[Corollary 2.2]{Thorpe2017IntroductionTO}) so that
    \begin{align*}
        \mL(\wh Y, Y) = \E_{\P_{\sX}}\bigg[\int_{\RR} \big|\Fc(t| X) - \F_{\smash{\wh \sY}|\sX}(t| X)\big|\, dt\bigg].
    \end{align*}
    Hence, quantile regression in $\mL$-loss is equivalent to the estimation of the conditional cdf. The loss $\mL$ is finite as long as both variables admit a first absolute moment and no thresholding is required on the range of estimated quantile levels. This makes $\mL$ a more convenient choice if compared to the sup-norm loss which has issues for quantiles close to $0$ and $1$.\\
    \subsection{Local polynomial estimators}
    The proposed estimator in this article is based on local polynomial estimators, which we briefly summarize. Let $\ell$ be an integer and recall that $m$ is the cardinality of $\{\balpha \in \{0, \dots, \ell\}^d: |\balpha| \leq \ell\}.$ Let $V(x) := (x^{\balpha}/\balpha!)_{0 \leq |\balpha| \leq \ell},$
    where the coordinates are ordered lexicographically on $\{0, \dots, \ell\}^d.$ The local polynomial estimator of order $\ell$ with kernel $K(u) = \mathds{1}(u \in \B(0, 1))$ is defined as
    \begin{align}
    \label{eq.theta.argmin}
        \wh \theta_{n}(x) \in \argmin_{\theta \in \RR^m}\sum_{i=1}^n\bigg[Y_i - \theta^TV\bigg(\frac{X_i - x}{h}\bigg)\bigg]^2\mathds{1}(X_i \in \B(x, h)).
    \end{align}
    In the case $\ell = 0,$ the above reduces to the Nadaraya-Watson estimator. The generalisation to $\ell > 1$ allows to exploit higher order smoothness of the unknown signal. More precisely, if the data generating process satisfies $Y_i = f(X_i)$ with $f$ sufficiently smooth, then, intuitively, solutions to the minimisation problem \eqref{eq.theta.argmin} should be close to the Taylor expansion of $f$ at $x.$ The choice of the kernel $K(u) = \mathds{1}(u \in \B(0, 1))$ simplifies the analysis. Following the same steps as in \cite[Section 1.6]{MR2724359}, we see that whenever the matrix
    \begin{align*}
        D_n(x) &:= \frac{1}{nh^d} \sum_{i=1}^n V\bigg(\frac{X_i - x}{h}\bigg)V^T\bigg(\frac{X_i - x}{h}\bigg)\mathds{1}(X_i \in \B(x, h))
    \end{align*}
    is invertible, the solution $\wh \theta_{n}(x)$ is unique. Its first coordinate $\wh \theta_{n, 1}(x)$ provides a linear estimator of the conditional expectation of $Y$ given $X = x$ and admits the following expression.
    \begin{align}
        \label{eq.def.theta.1}
        \wh \theta_{n, 1}(x) = \sum_{i=1}^n Y_i W_{n, i}(x),    
    \end{align}
    where the scalars $W_{n, i}(x)$ are called $\LP(\ell)$ weights and are expressed as
    \begin{align}
        \label{eq.def.weights}
        W_{n,i}(x) := \frac{1}{nh^d}V^T(0)D_n(x)^{-1}V\bigg(\frac{X_i - x}{h}\bigg)\mathds{1}\big(X_i \in \B(x, h)\big).
    \end{align}
    The linearity of $\wh \theta_{n, 1}$ in the $\{Y_i\}_{i=1}^n$ suggests to consider $\sum_{i=1}^n \mathds{1}(Y_i \leq t)W_{n, i}(x)$ as an estimator of $\Fc(t, x).$ However, since we consider a global loss function, the estimator must be well-defined on the whole support of $X,$ while $D_n(x)$ will likely fail to be invertible (almost) everywhere. To deal with this we introduce a thresholding step in the construction of the estimator. To motivate the approach, observe that by the law of large numbers, $D_n(x)$ is close to its expectation. Assuming that $\P_{\sX}$ admits a density $p$ and applying a change of variables leads to
    \begin{align}
        \label{eq.D}
        \begin{split}
        \E[D_n(x)] &= \frac 1{h^d}\int_{[x - h, x+h]^d}V^T\bigg(\frac{t - x}{h}\bigg)V\bigg(\frac{t - x}{h}\bigg)\, d\P_{\sX}(t)\\
        &= \int_{[-1, 1]^d}V^T(t)V(t)p(x + th)\, dt,
        \end{split}
    \end{align}
    If the density $p$ is bounded away from zero, then, using the fact that $[x - h, x+h]^d\cap[0, 1]^d$ has Lebesgue measure at least $h/2^d,$ we can relate the smallest eigenvalue of $\E[D_n(x)]$ to the smallest eigenvalue of
    \begin{align}
        D := \int_{[0, 1]^d}V(v)V^T(v)\, dv.
        \label{eq.D_def}
    \end{align}
    The smallest eigenvalue of $D$ is strictly positive since for $z\neq 0,$ $v\mapsto \|z^TV(v)\|_2^2$ is a non-zero polynomial and thus $z^TDz=\int_{[0, 1]^d} \|z^TV(v)\|_2^2 \, dv >0.$ Concentration arguments then show that the smallest eigenvalue $\lambda_1(D_n(x))$ of $D_n(x)$ is bounded away from zero on a large subset of $[0, 1]^d.$ The threshold $\lambda_1^{thr} >0$ allows us to control the size of this subset. This motivates the use of the estimator
    \begin{align}
        \label{eq.estimator}
        \wh \F_{\sY|\sX}(t|x) := \begin{dcases}
           \sum_{i=1}^n \mathds{1}(Y_i \leq t)W_{n,i}(x) & \text{if } \lambda_1(D_n(x)) \geq \lambda_1^{thr}, \text{ and}\\
            \mathds{1}(t \geq 0) &\text{otherwise.}
        \end{dcases}
    \end{align}
    Estimating the conditional cdf requires additional regularity assumptions, which we introduce in the following subsection.
    \subsection{Smoothness and tail assumptions}
    We consider a form of smoothness classes described in the following definition.
    \begin{definition}[Smoothness class]
        \label{def.smooth}
        Let $K > 0$ and $\beta = \ell + \gamma$ where $\ell$ is an integer and $\gamma \in (0, 1].$ We define the class $\mD(K, \beta)$ of conditional cumulative distribution functions $\Fc$ such that 
        \begin{enumerate}
            \item [$(i)$] for any $t,$ the partial derivatives $\partial^{\balpha} \Fc(t|x)$ exist for all $\balpha$ with $|\balpha|\leq \ell$ (recalling that derivatives are with respect to $x\mapsto \Fc(t|x)$), and 
            \item [$(ii)$] for any $|\balpha| = \ell$ and any $x, y \in [0, 1]^d,$
        \begin{align*}
            \int_{\RR} \big|\partial^{\balpha} \Fc(t|x) - \partial^{\balpha}\Fc(t|y)\big|\, dt \leq K\|x - y\|^\gamma.
        \end{align*}
        \end{enumerate}
    \end{definition}
    The smoothness property $(ii)$ of Definition \ref{def.smooth} adds an integral with respect to $t$ if compared to the more standard condition $|\partial^{\balpha}\Fc(t|x) - \partial^{\balpha}\Fc(t|y)| \leq K\|x - y\|^\gamma$ encountered in e.g.\ \cite{MR4047593} or \cite{padilla2022quantile}. Taking the square inside the integral in $(ii)$ leads to the assumption made in \cite{PIC20231564}. Their analysis is mentioned in more detail in Section \ref{sec.3}.
    In addition to the smoothness of the cdf, we impose two conditions on the distribution of $(X, Y).$
    \begin{assumption}
    \label{ass.1}
    \begin{enumerate}
        \item [$(i)$] There exists $C, \delta> 0$ such that the distribution of the pair $(X, Y)$ satisfies $\E[|Y|^{2 + \delta}|X = x] \leq C$ for all $x \in [0, 1]^d,$ and 
        \item [$(ii)$] the marginal distribution $\P_{\sX}$ admits a density $p$ that satisfies $\ul p \leq p(x) \leq \ol p$ \ for all $x \in [0, 1]^d,$ where $0 < \ul p \leq \ol p <\infty.$
    \end{enumerate}
    \end{assumption} 
    We denote by \[\mP=\mP(K, \beta, C, \delta, \ul p, \ol p)\] the class of pairs $(X, Y)$ that satisfy Assumption \ref{ass.1} with constant $C, \delta, \ul p, \ol p,$ and whose conditional cdf belongs to $\mD(K, \beta).$
\section{Rates of convergence}
    \label{sec.3}
    In the first subsection, we discuss the upper bounds on the minimax risk induced by the $\mL$-loss over the aforementioned smoothness classes. The second subsection shows that our problem is reducible to nonparametric regression in $L^1$-norm, leading to a matching lower bound for the minimax risk. 
    \subsection{Upper bound}

    We are now ready to state our first result, where the estimator $\wh \F_{\sY|\sX}$ is used to derive an upper bound for the risk associated with $\mL$ over the class $\mP.$
    \begin{theorem}
        \label{th.main.theorem}
        Let $K, C, \delta > 0, \ 0 < \ul p \leq \ol p < \infty$ and $\beta = \ell + \gamma$ with integer $\ell$ and $\gamma \in (0, 1].$ Then, there exists a constant $C_0 > 0$ such that
        \begin{align*}
            \inf_{\wh \F_{\sY|\sX}} \  \sup_{(X, Y) \in \mP(K, \beta, C, \delta, \ul p, \ol p)}\E\bigg[\int_{\RR} \big|\Fc(t|X) - \wh \F_{\sY|\sX}(t|X)\big|\, dt\bigg] &\leq C_0n^{-\beta/(2\beta + d)},
        \end{align*}
        where the infimum is taken over all estimators and $C_0$ depends only on $K, \beta, C, \delta,$ $\ul p, \ol p, D$ and $d.$ In particular, this upper bound holds for the estimator \eqref{eq.estimator} with the choice $\lambda_1^{thr} = \ul p\lambda_1(D)/2,$ and $h = n^{-1/(2\beta + d)}.$
    \end{theorem}
    In the proof of Theorem \ref{th.main.theorem}, condition $(i)$ of Assumption \ref{ass.1} is sufficient to control the variance of $\wh \F_{\sY|\sX}.$ Smoothness gives control over the bias term. Assumption \ref{ass.1} $(ii)$ is used throughout the proof and could potentially be weakened. This assumption is nonetheless standard in nonparametric regression. Theorem \ref{th.main.theorem} shows that our setting leads to generative rates of the same order as in nonparametric regression. We show in the next section that those rates are actually minimax optimal.\\ 

    The analysis in this article shares similarities with several other works. Up to the thresholding, the estimator \eqref{eq.estimator} seems to have first appeared in \cite{hall1999methods} in the case $\ell \in \{0, 1\}$ and for general kernel function. Slight variations of the latter have been used for quantile regression, where the estimator usually undergoes a specific inversion scheme to approximate the quantile function. This avoids the commonly known crossing quantile curves phenomenon, where a quantile regressor might not be monotone in $\tau.$ See, e.g., \cite{yu1998local} for $d = \ell = 1$ or \cite{dette2008non} for $d = 1$ and $\ell \in \{0, 1\}$ as well as the references therein. In particular, the results of \cite{yu1998local, dette2008non} concern asymptotic normality and pointwise mean squared error at fixed level $\tau.$ This is incomparable with the global convergence rates obtained in our analysis. Additionally, the crossing quantile curves phenomenon is an issue for inference but does not impact generative tasks as long as the generated variable mimics the original variable well enough.  More recently, \cite{AOS2320} uses \eqref{eq.estimator} with $\ell = 0$ to estimate fixed quantiles of heavy-tailed conditional distributions. Departing from the quantile regression setting, global convergence rates for thresholded local polynomial estimators seem to have been achieved for the first time in \cite{chhor2024benign} for interpolating regressors and the $L^2$ risk. The very recent and closely related work \cite{PIC20231564} analyses \eqref{eq.estimator} for $\ell = 0$ in the context of distributional regression. They consider the risk $\E\big[\int (\Fc(t|X) - \wh \F_{\sY|\sX}(t|X))^2\, dt\big]$ and derive minimax rates with respect to the class $\mD(K, \beta)$ with $\beta \in (0, 1],$ and as mentioned below Definition \ref{def.smooth}, the smoothness assumption $(ii)$ in Definition \ref{def.smooth} replaced by
    \begin{align*}
        \int_{\RR} \big(\Fc(t|x) - \Fc(t|y)\big)^2\, dt \leq K\|x - y\|^{2\beta}.
    \end{align*}
    We stress that extending their analysis to the case $\beta > 1$ can be obtained from the arguments used in our upper and lower bounds.\\
    \subsection{Lower bound}
    Let $f_0 \colon x \mapsto \E[Y|X = x]$ denote the true conditional mean. From a generative procedure $\wh Y$ one can obtain the estimator $\wh f \colon x \mapsto \E[\wh Y|X = x, S_n],$ where, as before, $S_n$ denotes the sample. Applying Jensen's inequality to the definition of $\mL$ leads to the reduction
    \begin{align}
    \begin{split}
        \label{eq.reduction}
        \E[\mL(\wh Y, Y)] &= \E\big[|\wh Y - Y|\big]\\
        &\geq \E\big[\big|\E[\wh Y|X, S_n] - \E[Y|X]\big|\big]\\
        &= \E\big[\| f_0 - \wh f\|_{L^1(\P_{\sX})}\big].
    \end{split}
    \end{align}
    Hence, the generative problem in $\mL$ loss is at least as hard as a regression problem in $L^1(\P_{\sX})$ loss. The exact setting of the corresponding regression problem is, however, unknown as the definition of the class $\mD(L, \beta)$ does not yield immediate information on the smoothness of $f_0$ and the behaviour of the residuals $Y - f_0(X).$ We now relate the two problems and obtain a lower bound that matches the upper bound obtained in Theorem \ref{th.main.theorem}.\\
    The upper bound is achieved over all distributions of $(X, Y)$ satisfying Assumption \ref{ass.1} and Definition \ref{def.smooth}. For the lower bound, it is sufficient to consider a submodel. Consider nonparametric regression with Gaussian noise $\eps \sim \mN(0, \sigma^2)$ that is independent of $X$ and assume there exists a function $f_0$ such that the data generating process satisfies
    \begin{align}
        \label{eq.regression.model}
        Y = f_0(X) + \eps.
    \end{align}
    Then, $f_0(x)=\E[Y|X = x]$ and thanks to \eqref{eq.reduction}, we obtain a lower bound for the generative risk if we can derive a lower bound for estimating the underlying regression function $f_0.$ 

    While the upper bound is taken over the class $\mD(K, \beta)$ of conditional cumulative distribution functions $\Fc$ introduced in Definition \ref{def.smooth}, 
    we will now show that for the lower bounds it is enough to assume that $f_0$ lies in a H\"older class
    \begin{align*}
        \mH(L, \beta) := \Bigg\{f \colon \RR^d \to \RR: \sum_{0 \leq |\balpha| \leq \ell-1} \big\|\partial^{\balpha}f\|_{\infty} + \sup_{\substack{ |\balpha| = \ell \\ x \neq y}}\frac{\big|\partial^{\balpha}f(x) - \partial^{\balpha} f(y)\big|}{\|x - y\|^{\gamma}} \leq L\Bigg\}.
    \end{align*}
    Compared to the generative setting of Section \ref{sec.2}, regression only requires estimating the conditional mean of $Y.$ Model \eqref{eq.regression.model} corresponds to a particular case of Model \eqref{eq.generative.model}, where the distribution of $Y|X$ is known up to its mean. Indeed, in the regression model, the conditional cdf of $Y$ is given by 
    \begin{align*}
        \Fc(t|x) = \Phi\bigg(\frac{t - f_0(x)}{\sigma}\bigg),
    \end{align*}
    where $\Phi$ denotes cdf of $\mN(0, 1).$  Therefore, estimating $f_0$ can be achieved by considering Model \eqref{eq.generative.model} over the class 
    \begin{align*}
        \mD'(L, \beta, \sigma) := \bigg\{\F\colon (x, t) \mapsto \Phi\bigg(\frac{t - f(x)}{\sigma}\bigg) \ \text{ for } f \in \mH(L, \beta)\bigg\}.
    \end{align*}
    At this point, taking the supremum over the class $\mH(L, \beta)$ and the infimum over all estimators in \eqref{eq.reduction} shows that
    \begin{align}
    \label{eq.reduction2}
        \inf_{\wh \F_{\sY|\sX}}\sup_{\Fc \in \mD'(L, \beta, \sigma)} \E\big[\mL(\wh Y, Y)\big] &\geq \inf_{\wh f}\sup_{f_0 \in \mH(L, \beta)}\E\big[\|\wh f - f_0\|_{L^1(\P_{\sX})}\big].
    \end{align}
    where both infimas are taken over all estimators, and the right-hand side assumes Model \eqref{eq.regression.model} with noise variance $\sigma^2.$ The next result relates the smoothness classes $\mD'(L, \beta)$ and $\mD(K, \beta).$
    \begin{proposition}
        \label{prop.smoothness.is.satisfied}
        If $\beta, \sigma >0,$ then for all $K > 0$ there exists $L = L(K, \beta, \sigma) > 0$ such that
        \begin{align*}
            \mD'(L, \beta, \sigma) \subseteq \mD(K, \beta),
        \end{align*}
    \end{proposition}
    Applying Proposition \ref{prop.smoothness.is.satisfied} to \eqref{eq.reduction2} shows the following general lower bound.
    \begin{theorem}
    \label{th.general.lower.bound}
        If $X \sim \P_{\sX}$ is supported on $[0, 1]^d,$ then, for all $K, \beta> 0,$ there exists $L, \sigma > 0$ such that
        \begin{align*}
            \inf_{\wh \F_{\sY|\sX}} \ \sup_{\Fc \in \mD(K, \beta)}\E\bigg[\int_{\RR} \big|\Fc(t|X) - \wh \F_{\sY|\sX}&(t|X)\big|\, dt\bigg]\\
            &\geq \inf_{\wh f}\sup_{f_0 \in \mH(L, \beta)}\E\Big[\|\wh f - f_0\|_{L^1(\P_{\sX})}\Big],
        \end{align*}
        where both infima are taken over all estimators, and the right-hand side assumes Model \eqref{eq.regression.model} with noise variance $\sigma^2.$
    \end{theorem}
    We now additionally assume that the marginal distribution $\P_{\sX}$ satisfies Assumption \ref{ass.1} $(ii),$ that is, its density $p$ satisfies $\ul p \leq p(x) \leq \ol p$ for all $x$ in $[0, 1]^d.$ The minimax risk of estimating $f_0\in \mH(L, \beta)$ in $L^1(\P_{\sX})$ norm is then shown to be lower bounded by a multiple of $n^{-\beta/(2\beta + d)}$ in Proposition \ref{prop.lower.bounds.regression} in the appendix. In the previous theorem, the supremum on the left-hand side is taken with respect to the class $\mD(K, \beta),$ whereas the upper bound in Theorem \ref{th.main.theorem} is derived for the smaller class $\mP=\mP(K, \beta, C, \delta, \ul p, \ol p).$ To complete the derivation of the lower bound, requires checking that a careful choice of $L, \sigma$ leads to Assumption \ref{ass.1} $(i)$ being satisfied for all pairs $(X, Y)$ distributed according to Model \eqref{eq.regression.model}. This last step follows from the following observation: If $Y = f_0(X) + \eps,$ then, for all $\delta > 0,$ it holds that 
    \begin{align*}
        \sup_{x \in [0, 1]^d}\E[|Y|^{2 + \delta}|X = x] &\leq 2^{1 + \delta}\big(L^{2 + \delta} + \sigma^{2 + \delta}\E[|\eps_0|^{2 + \delta}]\big),
    \end{align*}
    where $\eps_0 \sim \mN(0, 1).$ Consequently, for all $C, \delta > 0,$ one can pick $\sigma = \sigma(C, \delta)$ satisfying $\sigma \leq (C/(2^{2 + \delta}\E[|\eps_0|^{2 + \delta}]))^{1/(2 + \delta)}.$ Further, making use of Proposition \ref{prop.smoothness.is.satisfied}, one can pick $L = L(K, \beta, C, \delta, \sigma) > 0$ small enough so that on the one hand, $L^{2 + \delta} \leq C/2^{2 + \delta}$ and, on the other hand, $\mD'(L, \beta, \sigma) \subseteq \mD(K, \beta).$ This ensures that $(X, Y) \in \mP(K, \beta, C, \delta, \ul p, \ol p)$ and immediately proves our last result.
    \begin{theorem}
        \label{th.lower.bound.generative}
        Work under the assumptions of Theorem \ref{th.main.theorem}. Then there exists a strictly positive constant $c = c(K, \beta, \ul p, \ol p, d, \delta, C)$ such that
    \begin{align*}
        \inf_{\wh \F_{\sY|\sX}}\sup_{(X, Y) \in \mP(K, \beta, C, \delta, \ul p, \ol p)}\E\bigg[\int_{\RR} \big|\Fc(t|X) - \wh \F_{\sY|\sX}(t|X)\big|\, dt\bigg] \geq cn^{-\beta/(2\beta + d)},
    \end{align*}
    where the infimum is taken over all estimators.
    \end{theorem}
    This provides a matching lower bound to Theorem \ref{th.main.theorem}, shows that the estimator \eqref{eq.estimator} is minimax rate optimal.
\section{Conclusion}
    \label{sec.discussion}
    We have shown how the problem of conditional generative modelling is reduced to quantile regression for univariate $Y$ and have derived the corresponding minimax-optimal rates of convergence under smoothness assumptions. A natural future direction is to consider the case where $Y$ is multivariate. In general dimension and under mild conditions, the Knothe-Rosenblatt map generalizes the quantile transform. Nevertheless, it is hard to predict whether the Knothe-Rosenblatt map leads to minimax optimal rates and which assumptions on the generative procedure need to be imposed.

\appendix
\section{Proofs}\label{appn} 

        \noindent\textbf{Additional notation.} For $x_1, \dots, x_n \in \RR^d,$ we write $x^n := (x_1, \dots, x_n).$ This is a $n\times d$ matrix. Similarly, for a collection of $\RR^d$-valued random variables $X_1, \dots, X_n,$ we set $X^n := (X_1, \dots, X_n).$ The $m$-hypersphere in $\ell_2$ norm is denoted by $S^m.$ The operator norm associated to the Euclidean norm is denoted by $\|\cdot\|_{op},$ that is, $\|A\|_{op} = \sup_{\|x\| = 1}\|Ax\|.$ The binomial distribution with $n$ trials and success probability $p$ is denoted by $\Bin(n, p).$ Finally, $\mB(\RR^p)$ denotes the Borel sigma-algebra generated by the open subsets of $\RR^p.$

        \subsection{Proof of the upper bound}
         Before proving our upper bound, we state an extension of \cite[Proposition 1.12]{MR2724359} to the multivariate case $d \geq 1.$
        \begin{lemma}
            \label{lem.LP.weights}
            Let $x, x_1, \dots, x_n \in \RR^{d}.$ If the matrix $D_n(x)$ is invertible and $W_{n, i}$ denote the  $\LP(\ell)$ weights defined in \eqref{eq.def.weights}, then, for any $\ell$-th degree polynomial $q(x) = \sum_{0 \leq |\balpha| \leq \ell} A_{\balpha}x^{\balpha},$ with real coefficients $A_{\balpha},$ it holds that
            \begin{align*}
                \sum_{i=1}^n q(x_i)W_{n, i}(x) = q(x).
            \end{align*}
            In particular, if $x$ is a root of $q,$ then $\sum_{i=1}^n q(x_i)W_{n, i}(x) = 0.$
        \end{lemma}
        \begin{proof}
            Let $h > 0.$ Because $q$ is a polynomial of degree $\ell,$ its derivatives of order $|\balpha| >\ell$ vanish. This, and Taylor's theorem applied to $q$ shows that for all $x, y \in \RR^d,$
            \begin{align*}
                q(y) &= \sum_{0 \leq |\balpha| \leq \ell} \frac{(y - x)^{\balpha}}{\balpha!}\partial^{\balpha}q(x)\\
                &= \sum_{0 \leq |\balpha| \leq \ell} \frac{(y - x)^{\balpha}}{h^{|\balpha|}\balpha!}h^{|\balpha|}\partial^{\balpha}q(x) = \mathbf{q}(x)^TV\bigg(\frac{y - x}{h}\bigg),
            \end{align*}
            where $\mathbf{q}(x) := (h^{|\balpha|}\partial^{\balpha}q(x))_{\smash{0 \leq |\balpha|\leq \ell}}$ and $V(x) = (x^{\balpha}/\balpha !)_{0\leq \balpha \leq \ell}.$
            Consider $\wh \theta_n(x)$ defined as in \eqref{eq.theta.argmin}. As mentioned in Section \ref{sec.2}, the local polynomial estimator is linear in the $Y_i,$ and the weights $W_{n, i}(x)$ do not depend on $Y_1, \dots, Y_n$ (as long as the matrix $D_n(x)$ is invertible). This means that the choice of $Y_i$ does not impact the properties of the $\LP(\ell)$ weights. Nevertheless, for the sake of this proof, we choose $Y_i = q(x_i)$ for all $1 \leq i \leq n.$ With this choice, the definition of $\wh \theta_n(x)$ reads
            \begin{align}
            \label{eq.previous.display}
                \begin{split}
                \wh \theta_n(x) &\in \argmin_{\theta \in \RR^m} \sum_{i=1}^n\bigg[q(x_i) - \theta^TV\bigg(\frac{x_i - x}{h}\bigg)\bigg]^2\mathds{1}(x_i \in \B(x, h))\\
                &= \argmin_{\theta \in \RR^m} \sum_{i=1}^n \bigg[\big(\mathbf{q}^T(x) - \theta^T\big)V\bigg(\frac{x_i - x}{h}\bigg)\bigg]^2\mathds{1}(x_i \in \B(x, h))\\
                &= \argmin_{\theta \in \RR^m} \big(\mathbf{q}(x) - \theta)^TD_n(x)\big(\mathbf{q}(x) - \theta\big),
                \end{split}
            \end{align}
            By definition, $D_n(x)$ is a Gram matrix. Hence, it is symmetric and positive semi-definite. By assumption, it is invertible, which means that $D_n(x)$ is positive definite. Hence, $(\mathbf{q}(x) - \theta)^TD_n(x)(\mathbf{q}(x) - \theta) \geq 0$ for all $\theta,$ with equality if and only if $\theta = \mathbf{q}(x).$ This shows that the $\argmin$ in \eqref{eq.previous.display} is a singleton and $\wh \theta_n(x) = \mathbf{q}(x).$

            In particular, the first coordinate $q(x)$ of $\mathbf{q}(x)$ is equal to the first coordinate $\wh \theta_{n, 1}(x)$ of $\wh \theta_n(x).$ Together with \eqref{eq.def.theta.1}, this leads to
            \begin{align*}
                q(x) = \sum_{i=1}^n q(x_i)W_{n, i}(x).
            \end{align*}
        \end{proof}
        Lemma \ref{lem.LP.weights} is particularly useful when dealing with Taylor polynomials. For a sufficiently smooth multivariate function $f\colon \RR^d \to \RR,$ its Taylor polynomial of degree $\ell \geq 0$ around $x$ is
        \begin{align*}
            y \mapsto \sum_{0 \leq |\balpha|\leq \ell} \frac{(y - x)^{\balpha}}{|\balpha|!}\partial^{\balpha}f(x).
        \end{align*}
        \begin{remark}[$\LP(\ell)$ weights and Taylor approximations]
        \label{rem.taylor.LP}
            For any $\ell$ times differentiable function $f \colon \RR^d \to \RR,$ if $q$ is the Taylor polynomial of degree $\ell$ of $f$ around $x$ and $R_{\ell}(f, x, x_i)$ is the remainder in the approximation of $f(x_i)$ by Taylor expansion of $f$ around $x,$ then
            \begin{align*}
                \sum_{i=1}^n f(x_i)W_{n, i}(x) &= \sum_{i=1}^n \Big[q(x_i) + R_{\ell}(f, x, x_i)\Big]W_{n, i}(x)\\
                &= q(x) + \sum_{i=1}^nR_\ell(f, x, x_i)W_{n, i}(x)\\
                &= f(x) + \sum_{i=1}^nR_\ell(f, x, x_i)W_{n, i}(x),
            \end{align*}
            where the second equality follows from Lemma \ref{lem.LP.weights}, and the last equality comes from the fact that $q(x) = f(x)$ by definition of the Taylor polynomial of $f$ around $x.$
        \end{remark}
        \begin{proof}[\textbf{\textit{Proof of Theorem \ref{th.main.theorem}}}]
            Recall that $\lambda_1^{thr} = \ul p \lambda_1(D)/2.$ Consider the random set $E_n = E_n(X_1, \dots, X_n) := \{x : \lambda_1(D_n(x)) \geq \lambda_1^{thr}\}.$  We denote $\mathds{1}_{E_n}(x) := \mathds{1}(x \in E_n).$ For $x \in E_n,$ we define
            \begin{align*}
                \ol \F_{\sY|\sX}(t|x) := \sum_{i=1}^n \Fc(t|X_i)W_{n, i}(x).
            \end{align*}
            We begin by decomposing the loss as follows
            \begin{align*}
                \E\bigg[\int_{\RR} \big|\Fc(t|X) - \wh \F_{\sY|\sX}(t|X)\big|\, dt\bigg] &\leq (a) + (b) + (c),
            \end{align*}
            where 
            \begin{align*}
                (a) &:= \E\bigg[\int_{\RR} \big|\Fc(t|X) - \ol \F_{\sY|\sX}(t|X)\big|\, dt\mathds{1}_{E_n}(X)\bigg],\\
                (b) &:= \E\bigg[\int_{\RR} \big|\ol \F_{\sY|\sX}(t|X) - \wh \F_{\sY|\sX}(t|X)\big|\, dt\mathds{1}_{E_n}(X)\bigg],\\
                (c) &:= \E\bigg[\mathds{1}(X \in E_n^c)\int_{\RR} \big|\Fc(t|x) - \mathds{1}(t \geq 0)\big|\, dt\bigg].
            \end{align*}
            We now proceed by bounding each term separately.\\
            \textbf{\underline{Bound for $(a)$}.} Set $(a^*) := \int|\Fc(t|x) - \sum_{i=1}^n W_{n,i}(x)\Fc(t|X_i)|\, dt$ and consider an arbitrary $x \in E_n.$ Applying for each $i$ Taylor's theorem with remainder in Lagrange's form shows that
            \begin{align*}
                (a^*) &= \int_{\RR}\Big|\sum_{i=1}^n W_{n, i}(x)\big(\Fc(t|x) - \Fc(t|X_i)\big)\Big|\, dt\\
                &= \int_{\RR} \Big|\sum_{i=1}^n \sum_{1 \leq |\balpha| \leq \ell-1} \frac{(X_i - x)^{\balpha}}{\balpha!}\partial^{\balpha}\Fc(t|x)W_{n,i}(x)\\
                &\quad + \sum_{i=1}^n\sum_{|\balpha| = \ell}\frac{(X_i - x)^{\balpha}}{\balpha !}\partial^{\balpha}\Fc(t|\xi_i)W_{n,i}(x)\Big|\, dt,
            \end{align*}
            where the $\xi_i$ lie in $\B(x, \|X_i - x\|)\subset \B(x, h)$ for all $i \in \{1, \dots, n\}.$ Since $x \in E_n,$ and the $W_{n, i}$ are $\LP(\ell)$ weights, we can apply Lemma \ref{lem.LP.weights} and Remark \ref{rem.taylor.LP} (with $f = \Fc(t|\cdot)$ and $q$ the $\ell$-th degree Taylor polynomial associated to $\Fc(t|\cdot)$ around $x$) to obtain,
            \begin{align*}
                \sum_{i = 1}^n \sum_{1 \leq |\balpha| \leq \ell} \frac{(X_i - x)^{\balpha}}{\balpha !}\partial^{\balpha}\Fc(t|x)W_{n,i}(x) &= \sum_{1 \leq |\balpha| \leq \ell} \frac{(x - x)^{\balpha}}{\balpha !}\partial^{\balpha}\Fc(t|x)=0.
            \end{align*}
            Subtracting the terms corresponding to $|\balpha| = \ell$ leads to
            \begin{align*}
                \sum_{i = 1}^n \sum_{1 \leq |\balpha| \leq \ell-1} \frac{(X_i - x)^{\balpha}}{\balpha !}\partial^{\balpha}\Fc(t|x)&W_{n,i}(x)\\ 
                &= - \sum_{|\balpha| = \ell}\frac{(X_i - x)^{\balpha}}{\balpha !} \partial^{\balpha} \Fc(t|x) W_{n,i}(x).
            \end{align*}
            Applying Definition \ref{def.smooth}, using that $W_{n,i}(x)=0$ if $\|x-X_i\|>h,$ and setting $A_0 := L\sum_{|\balpha| = \ell}(\balpha !)^{-1},$
            \begin{align*}
                (a^*) &= \int_{\RR} \Big|\sum_{i=1}^n \sum_{1 \leq |\balpha| \leq \ell-1} \frac{(X_i - x)^{\balpha}}{\balpha!}\partial^{\balpha}\Fc(t|x)W_{n,i}(x)\\
                &\quad + \sum_{i=1}^n\sum_{|\balpha| = \ell}\frac{(X_i - x)^{\balpha}}{\balpha !}\partial^{\balpha}\Fc(t|\xi_i)W_{n,i}(x)\bigg|\, dt\\
                &= \int_{\RR} \Big|\sum_{i=1}^n\sum_{|\balpha| = \ell}\frac{(X_i - x)^{\balpha}}{\balpha !}\big(\partial^{\balpha}\Fc(t|\xi_i) - \partial^{\balpha}\Fc(t|x)\big)W_{n,i}(x)\bigg|\, dt\\
                &= \int_{\RR} \bigg|\sum_{i=1}^n \sum_{|\balpha| = \ell}\frac{(X_i - x)^{\balpha}}{\balpha !}\big(\partial^{\balpha} \Fc(t|\xi_i) - \partial^{\balpha}\Fc(t|x)\big)\big|W_{n,i}(x)\big|\, dt\\
                &\leq \sum_{i=1}^n \sum_{|\balpha| = \ell} \frac{|(X_i - x)^{\balpha}|}{\balpha !}\Bigg[\int_{\RR} \big|\partial^{\balpha} \Fc(t|\xi_i) - \partial^{\balpha} \Fc(t|x)\big|\, dt\Bigg]\big|W_{n, i}(x)\big|\\
                &\leq \sum_{i=1}^n \sum_{|\balpha| = \ell}\frac{Lh^\beta}{\balpha !}\big|W_{n,i}(x)\big|\\
                &= \sum_{i=1}^n A_0h^\beta \big|W_{n,i}(x)\big|.
            \end{align*}
            Hence,
            \begin{align*}
                (a) &\leq \sum_{i=1}^nA_0h^\beta\E\bigg[\big|W_{n,i}(X)\big|\mathds{1}_{E_n}(X)\bigg].
            \end{align*}
            To bound the expectation of the right-hand side, let $\mathbf{1} := (1)_{i=1}^m.$ We first show a useful inequality for the absolute value of the weights $W_{n, i}.$ Using the Cauchy-Schwarz inequality, the fact that $V(0) = (1, 0, \dots, 0),$ the properties of the operator norm as well as the fact that all the entries of $V((y - x)/h)$ are in $[-1, 1],$ we obtain
            \begin{align*}
                \begin{split}
                &\bigg|V^T(0)D_n(x)^{-1}V\bigg(\frac{y - x}{h}\bigg)\bigg| \mathds{1}_{E_n}(x)\mathds{1}(y \in \B(x, h)) \\ &\leq \bigg\|D_n(x)^{-1}V\bigg(\frac{y - x}{h}\bigg)\bigg\|_2\underbrace{\big\|V(0)\big\|_2}_{=1} \mathds{1}_{E_n}(x)\mathds{1}(y \in \B(x, h)) \\
                &\leq \underbrace{\bigg\|V\bigg(\frac{y - x}{h}\bigg)\bigg\|_2}_{\leq \|V(\mathbf{1})\|_2}\underbrace{\|D_n(x)^{-1}\|_{op}}_{\leq \lambda_1(D_n(x))^{-1}} \mathds{1}_{E_n}(x)\mathds{1}(y \in \B(x, h))\\
                &\leq \frac{\|V(\mathbf{1})\|_2}{\lambda_1^{thr}} \mathds{1}(y \in \B(x, h)),
                \end{split}
            \end{align*}
            where the last inequality follows from the definition of $E_n.$ This proves
            \begin{align}
                \label{eq.**} \tag{$\ast\ast$}
                \begin{split}
                &\bigg|V^T(0)D_n(X)^{-1}V\bigg(\frac{X_i - X}{h}\bigg)\bigg|\mathds{1}_{E_n}(X)\mathds{1}(X_i \in \B(X, h)) \\ &\leq \frac{\|V(\mathbf{1})\|_2}{\lambda_1^{thr}}\mathds{1}(X_i \in \B(X, h)),
                \end{split}
            \end{align}
            Together with $\P_{\sX}\{\B(X, h)\} \leq \ol p h^d,$ we obtain for $r>0,$
            \begin{align}
            \begin{split}
                &\E\bigg[\big|W_{n,i}(X)\big|^r\mathds{1}_{E_n}(X)\bigg] \\
                 &= \E\Bigg[\frac 1{n^rh^{rd}} \bigg|V^T(0)D_n(X)^{-1}V\bigg(\frac{X_i - X}{h}\bigg)\mathds{1}(X_i \in \B(X, h))\bigg|^r\mathds{1}_{E_n}(X)\Bigg]\\
                &\leq \frac{\|V(\mathbf{1})\|_2^r}{n^rh^{rd}(\lambda_1^{thr})^r}\E\Big[\mathds{1}(X_i \in \B(X, h))\Big]\\
                &\leq \frac{\|V(\mathbf{1})\|_2^r}{n^rh^{rd}(\lambda_1^{thr})^r}\E\big[\P_{\sX}\{\B(X, h)\}\big]\\
                &\leq 
                \frac{\|V(\mathbf{1})\|_2^r}{n^rh^{rd}(\lambda_1^{thr})^r}\ol p h^d.
                \end{split}
                \label{eq.bjsdhchsdj}
            \end{align}
            To bound (a), we only need $r=1,$
            \begin{align*}
                (a) &\leq \sum_{i=1}^n \frac{A_0\|V(\mathbf{1})\|_2\ol p}{n\lambda_1^{thr}}h^\beta =  \frac{A_0\|V(\mathbf{1})\|_2\ol p}{\lambda_1^{thr}}h^{\beta} =: C_1h^{\beta}.
            \end{align*}
            \noindent \textbf{\underline{Bound for $(b)$.}} 
            Applying the law of total expectation yields
            \begin{align*}
                (b) &= \E\Bigg[\E\bigg[\int_{\RR}\big|\ol \F_{\sY|\sX}(t|X) - \wh \F_{\sY|\sX}(t|X)\big|\, dt\mathds{1}_{E_n}(X)\Big| X^n, X\bigg]\Bigg].
            \end{align*}
            In a next step, we bound the inner conditional expectation using Fubini's theorem. For this step, we make the dependence of $E_n$ on $X^n \in \RR^{d\times n}$ explicit by writing $E_n(X^n).$ For all $x^n \in \RR^{d\times n}$ and $x \in \RR^d,$ we can apply Fubini's theorem to the function $f_n(y_1, \dots, y_n, t) = |\sum_{i=1}^n(\mathds{1}(y_i \leq t) - \Fc(x_i, t))W_{n, i}(x)\mathds{1}(x \in E_n(x^n))|,$ which is $\mB(\RR^{d + 1})$-measurable as a composition of $|\cdot|$ with a sum of $n \ \mB(\RR\times \RR)$-measurable functions. This leads to
            \begin{multline}
                \label{eq.fubini}
                \E\bigg[\int_{\RR}\big|\ol \F_{\sY|\sX}(t|X) - \wh \F_{\sY|\sX}(t|X)\big|\mathds{1}_{E_n}(X)\, dt\Big| X^n, X\bigg] \\
                = \int_{\RR} \E\Big[\big|\ol \F_{\sY|\sX}(t|X) - \wh \F_{\sY|\sX}(t|X)\big|\mathds{1}_{E_n}(X)\Big| X^n, X\Big]\, dt.
            \end{multline}
            Denote by $(b^*):= \E\big[|\ol \F_{\sY|\sX}(t|X) - \wh \F_{\sY|\sX}(t|X)|\mathds{1}_{E_n}(x)\big| X^n, X\big].$ The inner expectation on the right-hand side of \eqref{eq.fubini} is bounded as follows. Letting $C_\delta := C^{1/(2 + \delta)},$
            \begin{align}
            \label{eq.bound.we.have.to.justify.a.lot}
                \begin{split}
                (b^*) &\leq \E\Big[\big|\ol \F_{\sY|\sX}(t|X) - \wh \F_{\sY|\sX}(t|X)\big|^2\mathds{1}_{E_n}(X)\Big| X^n, X\Big]^{1/2}\\
                &\overset{(i)}{=} \bigg(\sum_{i=1}^n \Fc(t|X_i)\big(1 - \Fc(t|X_i)\big)W_{n, i}(X)^2\bigg)^{1/2}\mathds{1}_{E_n}(X)\\
                &\overset{(ii)}{\leq} \sqrt{\sup_{x \in \B(X, h)\cap[0, 1]^d}\Fc(t|x)(1 - \Fc(t|x))}\Big(\sum_{i=1}^n W_{n,i}(X)^2\Big)^{1/2}\mathds{1}_{E_n}(X)\\
                &\overset{(iii)}{\leq} \big\|W_n(X)\big\|_2 \sup_{x \in \B(X, h)}\PP\big\{|Y| \geq |t|\big|X' = x\big\}^{1/2}\mathds{1}_{E_n}(X)\\
                &\overset{(iv)}{\leq} \big\|W_n(X)\big\|_2\bigg[\mathds{1}\big(|t| < C_{\delta}\big) + \frac{C}{|t|^{2 + \delta}}\mathds{1}\big(|t| \geq C_{\delta}\big)\bigg]^{1/2}\mathds{1}_{E_n}(X),
                \end{split}
            \end{align}
            where $W_n(x) := (W_{n, 1}(x), \dots, W_{n, n}(x))$ and $X'$ is an i.i.d.\ copy of $X.$ We now give details about the steps $(i), \ (ii),$ $(iii)$ and $(iv)$ in the previous display. Notice that $\wh \F_{\smash{\sY|\sX}}(t|X)$ can be seen as a weighted sum of independent Bernoulli random variables with success probability $\Fc(t|X_i)$ and evaluating its variance gives $(i).$ To obtain step $(ii),$ recall that $X_i \in [0, 1]^d$ a.s., and that whenever $X \in E_n,$ if $X_i \in \B(X, h)^c,$ then $W_{n, i}(X) = 0.$ This means that 
            \begin{align*}
                \Fc(t|X_i)(1 - \Fc(t|&X_i))W_{n, i}(X)^2\\
                &\leq \sup_{x \in \B(X, h)\cap[0, 1]^d}\Fc(t|x)(1 - \Fc(t|x))W_{n, i}(X)^2,
            \end{align*}
            and using the monotonicity of the square root function on $[0, +\infty)$ finishes step $(ii).$ For step $(iii),$ we use the fact that for all $x \in \RR^d,$
            \begin{align*}
                \Fc(t|x)\big(1 - \Fc(t|x)\big) &\leq \Fc(t|x)\wedge \big(1 - \Fc(t|x)\big)\\
                &\leq \Fc(t|x)\mathds{1}(t \leq 0) + \big(1 - \Fc(t|x)\big)\mathds{1}(t > 0)\\
                &\leq \PP\big\{|Y| \geq |t| \big |X = x\big\}(\mathds{1}(t \leq 0) + \mathds{1}(t > 0))\\
                &= \PP\big\{|Y| \geq |t| \big|X = x\big\}.
            \end{align*}
            This, and $\B(X, h) \cap[0, 1]^d \subseteq [0, 1]^d,$ show that
            \begin{align*}
                \sqrt{\sup_{x \in \B(X, h)\cap[0, 1]^d}\Fc(t|x)(1 - \Fc(t|x))} \leq \sqrt{\sup_{x \in [0, 1]^d}\PP\big\{|Y| \geq |t| \big| X' = x\big\}},
            \end{align*}
            which concludes step $(iii).$ Step $(iv)$ is obtained by distinguishing the cases $|t| \geq C_{\delta}$ and $|t| < C_{\delta},$ using the fact that $\PP\{|Y|^{2 + \delta}\geq |t| |X' = x\} \leq 1$ for all $x,$ applying Markov's inequality and using Assumption \ref{ass.1} $(i)$ to derive
            \begin{align*}
                \sup_{x \in \B(X, h)}\PP\big\{|Y| \geq |t| \big| X' = x\big\} &= \sup_{x \in [0, 1]^d}\PP\big\{|Y|^{2 + \delta} \geq |t|^{2 + \delta}\big| X' = x\big\}\mathds{1}\big(|t| \geq C_{\delta}\big)\\
                &\ + \sup_{x \in [0, 1]^d}\PP\big\{|Y|^{2 + \delta} \geq |t|^{2 + \delta}\big| X' = x\big\}\mathds{1}\big(|t| < C_{\delta}\big)\\
                &\leq \sup_{x \in [0, 1]^d} \frac{\E\big[|Y|^{2 + \delta} \big|X' = x\big]}{|t|^{2 + \delta}}\mathds{1}\big(|t| \geq C_{\delta}\big)\\
                &\quad + \mathds{1}\big(|t| < C_{\delta}\big)\\
                &\leq \frac{C}{|t|^{2 + \delta}}\mathds{1}\big(|t| \geq C_{\delta}\big) + \mathds{1}\big(|t| < C_{\delta}\big).
            \end{align*}
            We now continue with the bound \eqref{eq.bound.we.have.to.justify.a.lot}, which we first integrate with respect to $t$ before taking the expectation with respect to the sample and $X.$ This, Jensen's inequality, and \eqref{eq.bjsdhchsdj} with $r=2$ yield
            \begin{align}
                \label{eq.markov.integral.constant}
                \begin{split}
                (b) &= \E\bigg[\int_{\RR}\big|\ol \F_{\sY|\sX}(t|X) - \wh \F_{\sY|\sX}(t|X)\big|\, dt \mathds{1}_{E_n}(X)\bigg]\\
                &\leq 4C_\delta \E\bigg[\big\|W_n(X)\big\|_2\mathds{1}_{E_n}(X)\bigg] \\
                &\leq 4C_\delta \bigg(\E\bigg[\big\|W_n(X)\big\|^2_2\mathds{1}_{E_n}(X)\bigg]\bigg)^{1/2} \\
                &= 4C_\delta \bigg(\E\bigg[\sum_{i=1}^n \big|W_{n,i}(X)\big|^2\mathds{1}_{E_n}(X)\bigg]\bigg)^{1/2} \\
                &\leq 4C_\delta \frac{\|V(\mathbf{1})\|_2\sqrt{\ol p}}{\lambda_1^{thr} \sqrt{nh^d}}\\
                &=: \frac{C_2}{\sqrt{nh^d}}.
                \end{split}
            \end{align}
            \textbf{\underline{Bound for $(c).$}} 
            Let $x \in [0, 1]^d.$ By Assumption \ref{ass.1} $(i),$ $E[|Y|]<\infty$ and thus
            \begin{align*}
                \int_{\RR}\big|\Fc(t|x) - \mathds{1}(t \geq 0)\big|\, dt &\leq \int_0^\infty \PP\{|Y| \geq t|X = x\} \, dt = E[|Y|] <\infty.
            \end{align*}
            Therefore,
            \begin{align*}
                (c) = \E\bigg[\mathds{1}(X \in E_n^c)\int_{\RR} \big|\Fc(t|X) - \mathds{1}(t \geq 0)\big|\, dt\bigg] \leq E[|Y|]\E[\mathds{1}(X \in E_n^c)].
            \end{align*}
            To bound $\E[\mathds{1}(X \in E_n^c)] = \E[\E[\mathds{1}(X \in E_n^c)|X]],$ we will focus on the inner conditional expectation. Given $X = x,$ two scenarios are possible. In the first scenario, at least one sampled covariate $X_i$ lies in $\B(x, h)$ and
            \begin{align*}
            D_n(x) &= \frac{1}{nh^d} \sum_{i=1}^n V\bigg(\frac{X_i - x}{h}\bigg)V^T\bigg(\frac{X_i - x}{h}\bigg)\mathds{1}(X_i \in \B(x, h))
            \end{align*}
            has at least one non-zero term. We can then apply concentration results to bound the probability that the smallest eigenvalue of $D_n(x)$ is small. In the second scenario, no covariate $X_i$ lies in $\B(x, h),$ and $D_n(x) = 0.$ This event is strictly included in the event $\lambda_1(D_n(x)) = 0.$ We partition the event $E_n^c$ accordingly. Let $G_n := [0, 1]^d \setminus (\cup_{i=1}^n \B(X_i, h)).$ We have
            \begin{align*}
                \E[\mathds{1}(X \in E_n^c)|X = x] &= \E[\mathds{1}(x \in E_n^c)]\\
                &= \E[\mathds{1}(x \in E_n^c\cap G_n)] + \E[\mathds{1}(x \in E_n^c\setminus G_n)]\\
                &\leq \E\big[\mathds{1}(x \in G_n)\big] + \E\big[\mathds{1}(x \in E_n^c \setminus G_n)\big].
            \end{align*}
            On the event $\{x \in G_n\}$, all $n$ covariate vectors $X_1, \dots, X_n$ lie in $[0, 1]^d\setminus \B(x, h),$ and
            \begin{align*}
                \E\big[\mathds{1}(x \in G_n)\big] &\leq \E\big[(1 -  \P_{\sX}\{\B(x, h)\})^n\big]\\
                &\leq \E\big[\exp(-n\P_{\sX}\{\B(x, h)\})\big]\\
                & \leq \E\big[(n\P_{\sX}\{\B(x, h)\})^{-1}\big],
            \end{align*}
            where we used that for any $t \in [0, 1], \ (1 - t)^n \leq e^{-tn} \leq (nt)^{-1}.$ After taking the expectation with respect to $X \sim \P_{\sX},$ we are left with
            \begin{align}
                \label{eq.1}
                \E[\mathds{1}(X \in E_n^c)] \leq \underbrace{\E\big[(n\P_{\sX}\{\B(X, h)\})^{-1}\big]}_{(c^*)} + \underbrace{\E[\mathds{1}(X \in E_n^c \setminus G_n)]}_{(c^{**})}.
            \end{align}
            The terms $(c^*)$ and $(c^{**})$ are bounded separately.\\
            \underline{Bound for $(c^*)$.} We proceed, as in \cite[equation (5.1)]{Gyorfi2002ADT}. Consider $(z_1, \dots, z_{T})$ to be the centers of an $h/2$-covering of $\supp(p).$ For any $h \leq ,$ and thus for $h=n^{-\beta/(2\beta + d)},$ it holds that $T \leq \lceil2/h\rceil^d \leq 3^d/h^d$ and
            \begin{align}
            \label{eq.half.ball.trick}
            \begin{split}
                (c^*) = \frac 1n\int \frac{d\P_{\sX}(x)}{\P_{\sX}\{\B(x, h)\}} &\leq \frac1n\sum_{i=1}^T \int_{\B(z_i, h/2)} \frac{d\P_{\sX}(x)}{\P_{\sX}\{\B(x, h)\}}\\
                &\overset{(*)}{\leq} \frac 1n\sum_{i=1}^T \int_{\B(z_i, h/2)}\frac{d\P_{\sX}(x)}{\P_{\sX}\{\B(z_i, h/2)\}}\\
                &= \frac 1n\sum_{i=1}^T \frac{\P_{\sX}\{\B(z_i, h/2)\}}{\P_{\sX}\{\B(z_i, h/2)\}}\\
                &= \frac {T}n \leq \frac{3^d}{nh^d},
            \end{split}
            \end{align}
            where step $(*)$ follows from the fact that for all $1 \leq i \leq T,$ and for all $x \in \B(z_i, h/2),$ the ball $\B(z_i, h/2)$ is included in the ball $\B(x, h),$ and hence $\P_{\sX}\{\B(x, h)\} \geq \P_{\sX}\{\B(z_i, h/2)\}.$\\
            \underline{Bound for $(c^{**})$.}
            For this step, we first bound $\E[\mathds{1}(x \in E_n^c\setminus G_n)].$ Let $x \in G_n^c,$ \[M_n := D_n(x) \quad \text{and} \ \ M := \E[M_n],\] where the expectation is taken with respect to the sample $X_1, \dots, X_n.$ The matrix $M_n$ is always positive semidefinite since it is a Gram matrix. Hence, $M$ is positive semidefinite and by \eqref{eq.D}, $M=\int_{[-1, 1]^d}V(t)V^T(t) p(x + th)\, dt.$ Consequently, $M$ is also positive definite since for any unit vector $v$ with $\|v\|_2=1,$
            \begin{align}
                \label{eq.eigenvalues.M.D}
                \begin{split}
                v^TMv &= \int_{[-1, 1]^d}\|v^TV(t)\|_2^2p(x + th)\, dt\\
                &\geq \ul p\int_{[-1, 1]^d}\|v^TV(t)\|_2^2\mathds{1}(x + th \in [-1, 1]^d)\, dt\\
                &\geq \ul p\lambda_1(D)^2\int_{[-1, 1]^d} \mathds{1}(x + th \in [-1, 1]^d)\, dt\\
                &\geq 2^{-d}\ul p\lambda_1(D)^2 > 0,
                \end{split}
            \end{align}
            with $D$ as defined in \eqref{eq.D_def}. Therefore the smallest eigenvalue of $M$ is strictly positive. Weyl's inequality \cite[Theorem 4.3.1]{zbMATH06125590} states that $\lambda_1(M)\leq \lambda_1(M_n(x))+\|M_n -M\|_{op}.$ Thus, for $$t^* := \ul p\lambda_1(D)/2^{d+1} \leq \lambda_1(M)/2$$ (see \eqref{eq.eigenvalues.M.D} for the right-most inequality), we have shown the inclusion
            \begin{align}
                \label{eq.recap.1}
                \{\lambda_1(D_n(x)) < t^*\} \subseteq \{\|D_n(x) - E[D_n(x)]\|_{op} > t^*\}.
            \end{align}
            We now apply Theorem 1.4 in \cite{Tropp_2011}, which we state here in a slightly adapted form.

            \begin{theorem}[Theorem 1.4 in \cite{Tropp_2011}]
                \label{th.bernstein.eigenvalues}
                Consider a sequence $A_1, \dots, A_n$ of independent, random, symmetric $m \times m$ matrices. Assume that for all $1 \leq k \leq n,$ $\E[A_k] = 0$ and $\lambda_{m}(A_k) \leq R,$ almost surely. Then, for all $t \geq 0,$
                \begin{align*}
                    \PP\bigg\{\lambda_m\bigg(\sum_{k=1}^nA_k\bigg) \geq t\bigg\} \leq d\exp\bigg(-\frac{t^2/2}{\sigma^2 + Rt/3}\bigg),
                \end{align*}
                with $\sigma^2 :=\|\sum_{k=1}^n\E[A_k^2]\|_{op}.$
            \end{theorem}
            We now check that the assumptions of Theorem \ref{th.bernstein.eigenvalues} hold for $M_n.$ Recall that $m = \card\{\balpha \in \{0, \dots, \ell\}^d: |\balpha|\leq \ell\}$ does not depend on the sample size $n.$  By definition, $M_n$ is the average of Gram matrices of the form $V((X_i - x)/h)V^T((X_i - x)/h)$ over the $X_i$ that are $h$-close to $x.$ For such $X_i,$ the entries of $V((X_i - x)/h)$ are all upper-bounded by $1.$ Therefore, $M_n$ is a Gram matrix with all entries upper bounded by $1.$ Similarly, we also obtain 
            \begin{align*}
                \lambda_1(M) \leq m\max_{1 \leq i,j\leq m} |M_{ij}|= m\max_{1 \leq i,j \leq m}\big|\E[M_{n, ij}]\big| \leq m\E\Big[\max_{1 \leq i,j \leq m}|M_{n, ij}|\Big]\leq m.
            \end{align*}
            The (maximal) eigenvalue of $M_n$ satisfies 
            \begin{align}
            \label{eq.bound.eigenvalues}
                \lambda_m(M_n) = \|M_n\|_{op} \leq m\max_{1 \leq i,j \leq m} |M_{n, ij}| \leq m,
            \end{align}
            where $M_{n, ij}$ is the $i,j$-th entry of $M_n,$ and the first inequality is an immediate implication of the Gershgorin circle inequality. Therefore we also have $\|\E[M_n^2]\|_{op} \leq m^2$ and we can apply Theorem \ref{th.bernstein.eigenvalues} conditionally on $N = N(x) := \sum_{i=1}^n \mathds{1}(X_i \in \B(x, h)) \sim \Bin(n, p_x)$ where $p_x := \P_{\sX}\{\B(x, h)\}.$ Using moreover $k \geq 1$ and $t^*\leq \lambda_1(M)/2\leq m/2,$ we obtain
            \begin{align}
                \label{eq.applied.bernstein}
                \begin{split}
                \PP\bigg\{\|M_n - M\|_{op} \geq t^* \big| N = k\bigg\} &\leq m\exp\bigg(-\frac{(kt^*)^2}{2m^2 + 2mkt^*/3}\bigg)\\
                &\leq m\exp\bigg(-\frac{3(kt^*)^2}{7m^2k}\bigg)\\
                &\leq m\exp\bigg(-\frac{kt^{*2}}{4m^2}\bigg).
                \end{split}
            \end{align}
            Taking the expectation with respect to $N$ in \eqref{eq.applied.bernstein} and using that $(1-p_x+p_x e^t)^n$ is the moment generating function of $N \sim \Bin(n, p_x)$ leads to
            \begin{align*}
                \PP\bigg\{\|M_n - M\|_{op} \geq t^*\bigg\} &\leq m\E\bigg[\exp\bigg(-N\bigg(\frac{t^*}{2m}\bigg)^2\bigg)\bigg]\\
                &\leq m\bigg(1 - p_x + p_x\exp\bigg(-\frac{t^{*2}}{4m^2}\bigg)\bigg)^n\\
                &\leq m\bigg(1 - p_x\bigg(1 - \exp\bigg(-\frac{\ul p^2\lambda_1(D)^2}{2^{2(d + 2)}m^2}\bigg) \bigg)\bigg)^n\\
                &\leq m\bigg(1 -  h^d \underbrace{2^{-d} \ul p\bigg(1 - \exp\bigg(-\frac{\ul p^2\lambda_1(D)^2}{2^{2(d + 2)}m^2}\bigg) \bigg)}_{=: c_1}\bigg)^n\\
                &\leq m\exp(-nh^d c_1) \\
                &\leq \frac{m}{nh^dc_1}.
            \end{align*}
            This, combined with the law of total expectation and the inclusion \eqref{eq.recap.1} proves that (recall that $D_n$ also depends on $X_1, \dots, X_n$)
            \begin{align*}
                (c^{**}) &\leq \E[\mathds{1}(X \in E_n^c\setminus G_n)]\\
                &= \E\bigg[\PP\Big\{\big\{\lambda_1(D_n(X)) < t^*\big\} \cap \big\{X \in G_n^c\big\}\Big| X\Big\}\bigg]\\
                &\leq \E\bigg[\PP\Big\{\big\{\|D_n(X) - D(X)\|_{op} > t^*\big\} \cap \big\{X \in G_n^c\big\}\Big| X\Big\}\bigg]\\
                &\leq \frac{m}{nh^dc_1},
            \end{align*}
            which enables us to conclude with \eqref{eq.1} and \eqref{eq.half.ball.trick} that
            \begin{align*}
                (c) &\leq  2^d\E[|Y|]\frac{2 + m/c_1}{nh^d} =: \frac{C_3}{nh^d}.
            \end{align*}
            Finally, combining the bounds for $(a), \ (b)$ and $(c),$ and choosing $h = n^{-1/(2\beta + d)}$ shows that
            \begin{align*}
                \E\bigg[\int_{\RR} \big|\Fc(t|X) - \wh \F_{\sY|\sX}(t|X)\big|\, dt\bigg] &\leq (a) + (b) + (c)\\
                &\leq C_1 h^{\beta} + \frac{C_2}{\sqrt{n h^d}} + \frac{C_3}{nh^d}\\
                &\leq C_1h^{\beta} + \frac{C_2 + C_3}{\sqrt{nh^d}}\\
                &\leq \big(C_1\vee(C_2 + C_3)\big)n^{-\beta/(2\beta + d)}\\
                &=: C_0n^{-\beta/(2\beta + d)},
            \end{align*}
            completing the proof.            
        \end{proof}

        \subsection{Proofs for the lower bounds}
        Given a multi-index $\balpha = (\alpha_1, \dots, \alpha_d)$ in $\{0, \dots, \ell\}^d$ and a $d$-dimensional vector $x = (x_1, \dots, x_d)$ we denote
            \begin{align*}
                x_{\balpha} := (\underbrace{x_1, \dots, x_1}_{\alpha_1\text{ times}}, \underbrace{x_2, \dots, x_2}_{\alpha_2\text{ times}}, \dots, \underbrace{x_d, \dots, x_d}_{\alpha_d\text{ times}}) \in \RR^{|\balpha|}.
            \end{align*}
            For a function $\phi \colon [0, 1]^d\times \RR \to \RR,$ we write
            \begin{align*}
                \phi_{\balpha}\colon (x_{\balpha}, t) \in [0, 1]^{|\balpha|}\times \RR \mapsto \phi(\mathbf{z}, t),
            \end{align*}
            where $\mathbf{z} = (z_1, \dots, z_d)$ with
            \begin{align*}
                z_i = \begin{dcases}
                    \frac{1}{\alpha_i}\sum_{j = \alpha_1 + \dots + \alpha_{i-1} + 1}^{\alpha_1 + \dots + \alpha_i} x_{\balpha, i} & \text{ if } \alpha_i > 0\\
                    x_i & \text{ otherwise.}
                \end{dcases} 
            \end{align*}
            The notation allows us to use the combinatorial interpretation of the Fa\`a di Bruno formula (see Proposition \ref{prop.faa.di.bruno}).  We emphasize that for all $x \in (0, 1)^d,$ $\phi_{\balpha}(x_{\balpha}, t) = \phi(x, t)$ and 
            \begin{align*}
                \frac{\partial^{|\balpha|}}{\partial x_{\balpha, 1}\dots \partial x_{\balpha, |\balpha|}}\phi_{\balpha}(x_{\balpha}, t) = \partial^{\balpha}\phi(x, t),
            \end{align*}
            meaning that, up to a change of representation, the functions $\phi$ and $\phi_{\balpha}$ are essentially the same. Finally, let $\bbeta = (\beta_1, \dots, \beta_d).$ We denote the componentwise difference by $\balpha - \bbeta.$ The notation $\balpha \leq \bbeta$ means that $\alpha_i \leq \beta_i$ for all $1 \leq i\leq d$ and the strict inequality $\balpha < \bbeta$ additionally means that $\alpha_i < \beta_i$ for at least one of the indexes $1 \leq i \leq d.$\\
            We now state an adapted version of the Fa\`a di Bruno formula in \cite{hardy2006combinatorics}.
        \begin{proposition}[Fa\`a di Bruno formula]
            \label{prop.faa.di.bruno}
            Let $\balpha \in \NN^d$, $g\colon \RR^d \to \RR,$ and $h\colon \RR \to \RR.$ If $h$ and $g$ admit a sufficient amount of derivatives, then
            \begin{align}
                \partial^{\balpha} (h\circ g)(x) = \sum_{\pi \in \Pi} h^{(|\pi|)}(g(x)) \prod_{B \in \pi} \partial^{B_*} g_{\balpha}(x_{\balpha}),
            \end{align}
            where 
            \begin{itemize}
                \item[-] \ $\pi$ runs through the set $\Pi$ of partitions of $\{1, \dots, |\balpha|\},$
                \item[-] \ for $B \in \pi, \ B_*$ denotes the multi-index $(\mathds{1}(i \in B))_{1 \leq i \leq |\balpha|},$
                \item[-] \ $\partial^{B_*}$ is the operator $\prod_{i=1}^{|\balpha|}\partial^{B_{*i}}/\partial x_{\balpha, i},$ and
                \item[-] \ $|S|$ denotes the cardinality of a set $S.$
            \end{itemize}    
        \end{proposition}
        From now on, $\Phi$ denotes the cdf of $\mN(0, 1),$ and $p_{\sigma}(t, \mu)$ denotes the density of $\mN(\mu, \sigma^2).$ In particular, we write $p(t) = p_1(t, 0)$ for the density of the standard normal distribution $\mN(0, 1).$ The proof of Proposition \ref{prop.smoothness.is.satisfied} relies on an intermediary result. For $k=0,1,\ldots,$ the Hermite polynomial of order $k$ is 
        \begin{align}
        \label{eq.hermite}
            H_k(t) := (-1)^{k}e^{t^2/2}\frac{d^k}{dt^k}e^{-t^2/2}.
        \end{align}
        In particular, $\Phi^{(k)}(t) = (-1)^{k-1}H_{k-1}(t)p(t).$

        \begin{lemma}
        \label{lem.prod.lipschitz}
            Let $L > 0,$ and $\beta = \ell + \gamma$ with integer $\ell$ and $\gamma \in (0, 1].$ If $f_1, \dots, f_k \in \mH(L, \beta),$ then for all $x, y \in [0, 1]^d,$ the function $\prod_{i=1}^k f_i$ satisfies
            \begin{align*}
                \bigg|\prod_{i=1}^kf_i(x) - \prod_{i=1}^k f_i(y)\bigg| \leq kL^k\|x - y\|^\gamma.
            \end{align*}
        \end{lemma}
        \begin{proof}
            Let $x, y \in [0, 1]^d,$ then
            \begin{align*}
                \bigg|\prod_{i=1}^k f_i(x) - \prod_{i=1}^k f_i(y)\bigg| \leq L^{k-1}\sum_{i=1}^k\big|f_i(x) - f_i(y)\big| \leq kL^k\|x - y\|^{1 + (\gamma - 1)\mathds{1}(\ell = 0)}.
            \end{align*}
            Finally, since $\gamma \leq 1$ and $\|x - y\| \leq 1,$ it holds that $\|x - y\| \leq \|x - y\|^\gamma,$ completing the proof.
        \end{proof}
        \begin{proposition}
        \label{prop.exp.reg.satisfies.ass}
            Let $\sigma > 0$ and $\beta = \ell + \gamma$ with integer $\ell$ and $\gamma \in (0, 1].$ If $f \in \mH(L, \beta)$ and $\Fc(t|x) = \Phi\big((t - f(x))/\sigma\big),$ then there exists a constant $K> 0$ such that
            \begin{align*}
                \max_{|\balpha| = \ell} \ \sup_{\substack{x, y \in [0, 1]^d\\ x\neq y}}\|x - y\|^{-\gamma} \int_{\RR} \big|\partial^{\balpha}\Fc(t|x) - \partial^{\balpha}\Fc(t|y)\big|\, dt \leq LK.
            \end{align*}
        \end{proposition}
        
        \begin{proof}
            Let $f \in \mH(L, \beta)$ and $ x, y \in [0, 1]^d.$ By assumption, $\Fc(t|x)=\Phi\big((t-f(x))/\sigma\big).$ 
            Recall that $p_\sigma(t, 0)$ denotes the density of $\mN(0, \sigma^2).$ We distinguish the cases $\beta \in (0, 1]$ and $\beta > 1.$\\
            \textbf{\underline{Case $\beta \in (0, 1].$}} We can directly use the mean value theorem followed by our assumption on $f$ to bound
            \begin{align*}
                \int_{\RR} \big|\Fc(t|x) - \Fc(t|y)\big|\, dt &\leq \int_{\RR}|f(x) - f(y)| \sup_{u \in [t - f(y), t- f(x)]}p_\sigma(u, 0)\, dt\\
                &\leq L\|x - y\|^\beta\int_{\RR}\sup_{u \in [t - L, t + L]}p_\sigma(u, 1)\, dt\\
                &= L\|x - y\|^\beta\bigg[\int_{t: |t| \leq L} p_{\sigma}(0, 0)\, dt\\
                &\qquad + 2\int_L^{\infty} p_{\sigma}(t - L, 0)\, dt\bigg]\\
                &\leq L\bigg(\frac{2L}{\sqrt{2\pi\sigma^2}}+1\bigg)\|x - y\|^\beta \\
                &=: LK\|x - y\|^\beta,
            \end{align*}
            where $K = \sqrt{2}L /\sqrt{\pi\sigma^2}+1$ does not depend on $(\beta,x, y).$ This proves the claim in the case $\beta \leq 1.$\\
            \textbf{\underline{Case $\beta > 1.$}} Let $\beta = \ell + \gamma > 1$ with integer $\ell \geq 1$ and $\gamma \in (0, 1].$ By assumption
            \begin{align*}
                \Fc(t|x) = \Phi\bigg(\frac{t - f(x)}{\sigma}\bigg).
            \end{align*}
            Set $g(x, t) := (t - f(x))/\sigma.$ Notice that $p_\sigma(t, f(x)) = \sigma^{-1}p(g(x, t)),$ and for all multi-indexes $\bbeta \in \{0, \dots, \ell\}^d,$ with $|\bbeta| \geq 1,$  $\partial^{\bbeta} g(x, t) = -\sigma^{-1}\partial^{\bbeta} f(x).$ Additionally, any partition $\pi$ contains at least one element. Applying Proposition \ref{prop.faa.di.bruno} and applying the Hermite polynomials defined in \eqref{eq.hermite} shows that
            \begin{align*}
                \partial^{\balpha}\Fc(t|x) &= \partial^{\balpha}(\Phi\circ g)(x, t)\\
                &= \sum_{\pi \in \Pi}\Phi^{(|\pi|)}(g(x, t))\prod_{B \in \pi}\partial^{B_*}g_{\balpha}(x_{\balpha}, t)\\
                &= \sum_{\pi \in \Pi}\Phi^{(|\pi|)}\bigg(\frac{t - f(x)}{\sigma}\bigg)\bigg(\frac{-1}\sigma\bigg)^{|\pi|}\prod_{B \in \pi}\partial^{B_*}f_{ \balpha}(x_{\balpha})\\
                &= p(g(x, t))\sum_{\pi \in \Pi}\frac{(-1)^{2|\pi|-1}}{\sigma^{|\pi|}}H_{|\pi|-1}\bigg(\frac{t - f(x)}{\sigma}\bigg)\prod_{B \in \pi}\partial^{B_*}f_{\balpha}(x_{\balpha})\\
                &= -p_\sigma(t, f(x))\underbrace{\sum_{\pi \in \Pi}H_{|\pi|-1}\bigg(\frac{t - f(x)}{\sigma}\bigg)\frac 1{\sigma^{|\pi|-1}}\prod_{B \in \pi}\partial^{B_*}f_{\balpha}(x_{\balpha})}_{=: Q(t, x)}.
            \end{align*}
            Using $|a_1b_1 - a_2b_2| \leq |a_1 - a_2||b_2| + |a_1||b_1 - b_2|,$ we obtain
            \begin{align}
            \label{eq.bound.delta.F}
            \begin{split}
                \big|\partial^{\balpha}\Fc(t|x) - \partial^{\balpha}\Fc(t|y)\big| &= \big|p_\sigma(t, f(x))Q(t, x) - p_\sigma(t, f(y))Q(t, y)\big|\\
                &\leq \underbrace{p_\sigma(t, f(x))\big|Q(t, x) - Q(t, y)\big|}_{(a)}\\
                &\quad + \underbrace{\big|p_\sigma(t, f(x)) - p_\sigma(t, f(y))\big||Q(t, y)|}_{(b)}.
            \end{split}
            \end{align}
            We now bound the terms $(a)$ and $(b)$ separately.\\
            \textbf{\underline{Bound for $(a).$}} Since $H_{|\pi|-1}$ is a Hermite polynomial, it follows that $t\mapsto Q(t, x)$ is a polynomial for any $x.$ Denote by $\{q_k(x)\}_{k=0}^{\ell}$ its coefficients so that
            \begin{align*}
                Q(t, x) = \sum_{k = 0}^{\ell} q_k(x)t^k.
            \end{align*}
            Each $q_k(x)$ is a polynomial in the partial derivatives up to order $\ell$ of $f$ at $x.$ Applying Lemma \ref{lem.prod.lipschitz} to each of the $q_k(x)$ ensures the existence of a constant $M$ independent of $x, t$ and $f$ such that for all $0 \leq k \leq \ell, \ |q_k(x) - q_k(y)| \leq ML\|x - y\|^\gamma.$ This, together with the triangle inequality and $P(|t|) := M\sum_{k=0}^{\ell} |t|^k,$ shows that
            \begin{align}
                \label{eq.bound.a}
                (a) = p_\sigma(t, f(x))\big|Q(t, x) - Q(t, y)\big| \leq p_\sigma(t, f(x))L\|x - y\|^\gamma P(|t|)
            \end{align}
            and concludes the bound for $(a).$\\
            \textbf{\underline{Bound for $(b).$}} We can write
            \begin{align*}
                \big|p_\sigma(t, f(x)) - p_\sigma(t, f(y))\big| &= p_\sigma(t, 0) \frac{|p_\sigma(t, f(x)) - p_\sigma(t, f(y))|}{p_{\sigma}(t, 0)}\\
                &= p_\sigma(t, 0)\bigg|\exp\bigg(\frac{2tf(x) - f(x)^2}{2\sigma^2}\bigg)\\
                &\qquad - \exp\bigg(\frac{2tf(y) - f(y)^2}{2\sigma^2}\bigg)\bigg|.
            \end{align*}
            The function $\mu \mapsto \exp\big((2t\mu - \mu^2)/(2\sigma^2)\big)$ is differentiable and its derivative is bounded by $(|t| + \mu)/\sigma^2.$ Its restriction to $[-L, L]$ is therefore Lipschitz with Lipschitz constant at most $(|t| + L)/\sigma^2.$ Consequently
            \begin{align*}
                \big|p_\sigma(t, f(x)) - p_\sigma(t, f(y))\big| &\leq p_\sigma(t, 0)|f(x) - f(y)|\frac{|t| + L}{\sigma^2}\\
                &\leq p_\sigma(t, 0)L\|x - y\|^\gamma\frac{|t| + L}{\sigma^2}.
            \end{align*}
            Recall that $Q(t, y) = \sum_{k=0}^{\ell} q_k(y) t^k$ and that the coefficients $q_k(y)$ are polynomials in the derivatives of $f$ up to order $\ell$ at $y.$ Since $f \in \mH(L, \beta),$ it holds that $|q_k(y)| \leq a_k,$ where $a_k$ is a linear combination of powers of $L$ and does not depend on $y.$ With $R(|t|) := ((|t| + L)/\sigma^2)\sum_{k=0}^{\ell} a_k|t|^k,$
            \begin{align}
                \label{eq.bound.b}
                (b) = \big|p_\sigma(t, f(x)) - p_\sigma(t, f(y))\big||Q_y(t)| \leq p_\sigma(t, 0)L\|x - y\|^\gamma R(|t|),
            \end{align}
            which concludes the bound for $(b).$ Plugging \eqref{eq.bound.a} and \eqref{eq.bound.b} into \eqref{eq.bound.delta.F} results in
            \begin{align*}
                \int_{\RR}\big|\partial^{\balpha}\Fc(t|x) -\partial^{\balpha} \Fc(t|y)\big|\, dt &\leq L\|x - y\|^\gamma\bigg[\int_{\RR} R(|t|)p_\sigma(t, 0)\, dt\\
                &\quad + \int_{\RR}P(|t|)p_\sigma(t, f(x))\, dt\bigg].
            \end{align*}
            Both integrals on the right-hand side are linear combinations of absolute moments of $\mN(0, \sigma^2)$ and $\mN(f(x), \sigma^2)$ with positive coefficients. Absolute moments of $\mN(\mu, \sigma^2)$ of all orders exist and are increasing functions of $|\mu|.$ Hence we can bound $\int_{\RR} P(|t|)p_{\sigma}(t, f(x)) \, dt\leq \int_{\RR} P(|t|)p_{\sigma}(t, L) \, dt.$ The latter is independent of $x$ and $y.$ Finally, since $x, y$ are arbitrary, and by eventually taking the max over $|\balpha| = \ell,$ we conclude that there exists a constant $K>0$ such that
            \begin{align*}
                \max_{|\balpha| = \ell} \ \sup_{\substack{x, y \in [0, 1]^d \\ x\neq y}}\|x - y\|^{-\gamma} \int_{\RR}\big|\partial^{\balpha} \Fc(t|x) - \partial^{\balpha} \Fc(t|y)\big|\, dt \leq LK.
            \end{align*}
        \end{proof}

        \begin{proof}[\textit{\textbf{Proof of Proposition \ref{prop.smoothness.is.satisfied}}}]
            The claim follows directly by applying Proposition \ref{prop.exp.reg.satisfies.ass} with $L$ small enough.
        \end{proof}

        For a self-contained presentation of the material, we include a lower bound for nonparametric regression under smoothness constraints with respect to the $L^1$ loss. While this is a known result, we could not find exactly this version in the literature. 
        \begin{proposition}
        \label{prop.lower.bounds.regression}
            Work under the assumptions of Model \eqref{eq.regression.model}. If $\P_{\sX}$ admits a density $p$ supported on $[0, 1]^d$ such that $0 < \ul p \leq p \leq \ol p < \infty$ for some $\ul p$ and $\ol p,$ then there exists $c = c(L, \beta, \sigma^2, \ul p, \ol p, d)$ such that
            \begin{align*}
                \inf_{\wh f} \sup_{f \in \mH(L, \beta)} \E\big[\|\wh f - f\|_{L^1(\P_{\sX})}\big] &\geq cn^{-\beta/(2\beta + d)},
            \end{align*}
            where the infimum is taken over all estimators.
        \end{proposition}

        \begin{proof}
            We wish to lower bound the quantity
            \begin{align}
                \label{eq.minimax1}
                \inf_{\wh f} \sup_{f \in \mH(L, \beta)} \E\big[\|\wh f - f\|_{L^1(\P_{\sX})}\big],
            \end{align}
            where the infimum is taken over all estimators. To do so, we will apply Assouad's lemma \cite[Lemma 2.12]{MR2724359}. Prior to that, we proceed by reducing our problem to the estimation of a vector in Hamming distance. For $p \in [1, \infty),$ denote by $\|\cdot\|_{L^p}$ the $L^p$ norm with respect to the Lebesgue measure on $[-1, 1]^d.$ Consider a kernel $\Phi \colon \RR^d \to \RR$ satisfying
            \begin{enumerate}
                \item [$(i)$] $\Phi$ is supported on $[-1, 1]^d,$
                \item [$(ii)$] $\Phi \in \mH(1, \beta),$
                \item [$(iii)$] $0 < \|\Phi\|_{L^1} =: C_{\Phi} < \infty,$ 
                \item [$(iv)$] $\|\Phi\|_{\infty} \leq 1.$
            \end{enumerate}
            To see that such kernel function exists, consider the univariate function $\psi \colon x \mapsto \exp(-1/(1 - x^2))\mathds{1}(|x| < 1).$ This function is infinitely smooth on the real line, and all its derivatives vanish outside of $(-1, 1).$ Now, consider the $d$-variate function $g(x) = \prod_{i=1}^d \psi(x_i).$ For any $\beta = \ell + \gamma,$ the quantity $K_{\beta} := \max\{\|\partial^{\balpha}g\|_{\infty} = \|\prod_{i=1}^d \psi^{(\alpha_i)}\|_{\infty}: |\balpha| \leq \ell\}$ is well-defined because the all the derivatives of $\psi$ are well-defined and bounded in infinite norm. Let $\Phi(x) := g(x)/(2\sqrt{d}K_{\beta}).$ This function trivially fulfils $(i)$ and $(iv),$ its derivatives are all at most $1$-Lipschitz up to order $\ell.$ Finally, for all $x, y \in [-1, 1]^d$ and all $|\balpha| = \ell,$ it holds that
            \begin{align*}
                \big|\partial^{\balpha}g(x) - \partial^{\balpha}g(y)\big|& \leq \frac{1}{2\sqrt{d}}\|x - y\|\\
                &\leq \frac{1}{2\sqrt{d}}\|x - y\| \mathds{1}(\|x - y\| \leq 1) + \mathds{1}(\|x - y\|>1)\\
                &\leq \frac{1}{2\sqrt{d}}\|x - y\|^\gamma\mathds{1}(\|x - y\|\leq 1) + \|x - y\|^\gamma\mathds{1}(\|x - y\|>1)\\
                &\leq \|x - y\|^\gamma.
            \end{align*}
            This shows that $\Phi$ is an example of such kernel function. Let $z_1, \dots, z_m$ be the centres of the elements of a packing of $[0, 1]^d$ by hypercubes of side-length $2h$ where $0 < h < 2^{-1/d}$ will be specified later. There exists a $c_P = c_P(d) > 0$ such that 
            \begin{align}
            \label{eq.bound.on.packing.number}
                c_P^{-1} h^{-d} \leq m \leq c_Ph^{-d}.
            \end{align}
            Consider $\Omega := \{0, 1\}^m$ and define $M = 2^m$ alternatives as follows: for all $w \in \Omega,$ set
            \begin{align*}
                f_w = \sum_{i=1}^M w_i \phi_i,
            \end{align*}
            where 
            \begin{align*}
                \phi_i(x) := Lh^{\beta}\Phi\bigg(\frac{x - z_i}{h}\bigg), \qquad 1 \leq i \leq m.
            \end{align*}
            Because the $z_i$ are $2h$-separated, and the $\phi_i$ are supported on $\B(z_i, h),$ the $\phi_i$ have disjoint support. Since moreover $\phi_i \in \mH(L, \beta)$ for all $1 \leq i \leq m,$ it holds that $f_w \in \mH(L, \beta)$ for all $w \in \Omega.$ We can now lower bound \eqref{eq.minimax1}:
            \begin{align}
                \label{eq.minimax2}
                \inf_{\wh f} \sup_{f \in \mH(L, \beta)} \E\big[\|\wh f - f\|_{L^1(\P_{\sX})}\big] &\geq \inf_{\wh f} \max_{w \in \Omega} \E\big[\|\wh f - f_w\|_{L^1(\P_{\sX})}\big].
            \end{align}
            Further, because the supports of the $\phi_i$ are mutually disjoint, for any estimator $\wh f$ and any $w \in \Omega,$
            \begin{align*}
                \big\|\wh f - f_w\big\|_{L^1(\P_{\sX})} &= \sum_{i = 1}^M \underbrace{\int_{\B_i} \big|\wh f(x) - w_i\phi_i(x)\big|\, d\P_{\sX}(x)}_{=: T_i(w_i, \wh f)},
            \end{align*}
            where $\B_i$ denotes the ball $\B(z_i, h).$ This leads, just as in Example 2.2 in \cite{MR2724359}, to define $\wh w_i := \argmin_{t \in \{0, 1\}}T_i(t, \wh f).$ Let $\wh w := (\wh w_1, \dots, \wh w_M).$ Using the triangle inequality and the definition of $T_i$ shows that
            \begin{align*}
                \big\|f_{\wh w} - f_w\big\|_{L^1(\P_{\sX})} &= \sum_{i=1}^M \int_{\B_i} \big|(\wh w_i - w_i)\phi_i(x)\big|\, d\P_{\sX}(x)\\
                &\leq \sum_{i=1}^M \bigg[\int_{\B_i} \big|\wh f(x) - w_i\phi_i(x)\big|\, d\P_{\sX}(x)\\
                &\quad + \int_{\B_i} \big|\wh f(x) - \wh w_i\phi_i(x)\big|\, d\P_{\sX}(x)\bigg]\\
                &= \sum_{i=1}^M \Big[T_i(w_i, \wh f) + T_i(\wh w_i, \wh f)\Big]\\
                &\leq 2\sum_{i=1}^M T_i(w_i, \wh f) = 2\big\|\wh f - f_w\big\|_{L^1(\P_{\sX})}.
            \end{align*}
            This last inequality and the mutually disjoint supports of the $\phi_i$'s enable us to reduce \eqref{eq.minimax2} further to
            \begin{align}
                \label{eq.red1}
                \inf_{\wh f} \max_{w \in \Omega} \E\big[\|\wh f - f_w\|_{L^1(\P_{\sX})}\big] \geq \frac 12\inf_{\wh w}\max_{w \in \Omega} \|f_{\wh w} - f_w\|_{L^1(\P_{\sX})}.
            \end{align}
            Next, we observe that
            \begin{align*}
                \|f_{\wh w} - f_w\|_{L^1(\P_{\sX})} &= \int \sum_{i=1}^m |\wh w_i - w_i|\phi_i(x)\, d\P_{\sX}(x)\\
                &= \sum_{i=1}^m|\wh w_i - w_i|\int \phi_i(x)\, d\P_{\sX}(x)\\
                &\geq \Ham(\wh w, w)\ul p \int \phi_1(x)\, dx = \Ham(\wh w, w)\ul p Lh^{\beta + d}\|\Phi\|_{L^1},
            \end{align*}
            where $\ul p$ denotes the assumed lower bound on the density of $\P_{\sX}$ on $[0,1]^d$ and $\Ham(\wh w, w)$ is the Hamming distance between $\wh w$ and $w$ defined as $\Ham(\wh w, w) := \sum_{i=1}^M \mathds{1}(\wh w_i \neq w_i).$ Applying this last inequality to the right-hand side of \eqref{eq.red1} shows that
            \begin{align*}
                \frac 12\inf_{\wh w}\max_{w \in \Omega} \|f_{\wh w} - f_w\|_{L^1(\P_{\sX})} \geq \frac{Lh^{\beta + d}\ul p}2\inf_{\wh w}\max_{w \in \Omega} \E\big[\Ham(\wh w, w)\big]\|\Phi\|_{L^1},
            \end{align*}
             We are now ready to apply Assouad's lemma in the Kullback-Leibler version \cite[equation (2.28)]{MR2724359},
            \begin{align*}
                \inf_{\wh w}\max_{w \in \Omega} \E\big[\Ham(\wh w, w)\big] &\geq \frac m2 \Big(e^{-\alpha}\vee (1 - \sqrt{\alpha/2})\Big),
            \end{align*}
            where $\alpha := \max_{w_1, w_2 \in \Omega : \Ham(w_1, w_2) = 1}\KL(\P_{w_1}, \P_{w_2})$ and $\P_{w}$ is the distribution of an i.i.d.\ sample $(X_1, Y_1), \dots, (X_n, Y_n)$ under model \eqref{eq.regression.model} when the underlying regression function is $f_w.$ Using $\KL(\mN(\mu_1, \sigma^2), \mN(\mu_2, \sigma^2)) = (\mu_1 - \mu_2)^2/(2\sigma^{2}),$ $\KL(\P^{\otimes n}, \Q^{\otimes n}) = n\KL(\P, \Q),$ and $\|\Phi\|_{L^2}^2\leq \|\Phi\|_{\infty} \|\Phi\|_{L^1}\leq \|\Phi\|_{L^1}$ lead to
            \begin{align*}
                \alpha &= \max_{\substack{w_1, w_2 \in \Omega\\ \Ham(w_1, w_2) = 1}}\frac{n\|f_{w_1} - f_{w_2}\|^2_{L^2(\P_{\sX})}}{2\sigma^2}\\
                &\leq \frac{n}{2\sigma^2}\max_{1 \leq i, j \leq m}\big\{\|\phi_i\|^2_{L^2(\P_{\sX})} + \|\phi_j\|_{L^2(\P_{\sX})}^2\big\}\\
                &= \frac{n}{\sigma^2}\max_{1 \leq i \leq m}\big\|\phi_i\big\|_{L^2(\P_{\sX})}^2 \leq \frac{nL^2h^{2\beta + d}\ol p}{\sigma^2}\|\Phi\|_{L^2}^2 \leq \frac{nL^2h^{2\beta + d}\ol p}{\sigma^2}\|\Phi\|_{L^1}.
            \end{align*}
            Pick $h = (\lambda L^2\|\Phi\|_{L^1} \ol p n/\sigma^2)^{-1/(2\beta + d)}$ with $\lambda > 0.$ Then $\alpha \leq \lambda.$ Recall that $c_P$ is a constant defined in \eqref{eq.bound.on.packing.number}. Combining all of the previous results leads to 
            \begin{align*}
                \inf_{\wh f} \sup_{f \in \mH(L, \beta)} \E\big[\|\wh f - f\|_{L^1(\P_{\sX})}\big] &\geq \frac{mLh^{\beta + d}\ul p\|\Phi\|_{L^1}}{4}\Big(e^{-\alpha}\vee (1 - \sqrt{\alpha/2})\Big)\\
                &\overset{(i)}{\geq} \frac{Lh^{\beta}\|\Phi\|_{L^1}\ul p}{4c_P}e^{-\alpha}\\
                &\geq \frac{L\|\Phi\|_{L^1}\ul p}{4c_P}\bigg(\frac{\sigma^2}{\lambda n L^2 \|\Phi\|_{L^1}^2\ol p}\bigg)^{\beta/(2\beta + d)}e^{-\lambda}.
            \end{align*}
            Step $(i)$ uses $m \geq c_P^{-1}h^{-d}$ from \eqref{eq.bound.on.packing.number}. Choosing $\lambda \geq 2^{(2\beta + d)/(\beta d)}\sigma^2/(L^2\|\Phi\|_{L^1}\ol p)$ guarantees that $h \leq 2^{-1/d}.$ Finally, for all $n\geq 1,$ it holds that there exists $c = c(L, \beta, \sigma^2, \ul p, \ol p, d) > 0$ such that
            \begin{align*}
                \inf_{\wh f} \sup_{f \in \mH(L, \beta)} \E\big[\|\wh f - f\|_{L^1(\P_{\sX})}\big] &\geq cn^{-\beta/(2\beta + d)}.
            \end{align*}
        \end{proof}

\section*{Acknowledgements}
    P.Z.\ would like to thank Mathias Trabs for an additional helpful reference.
\printbibliography
\end{document}